\documentclass[a4paper,12pt,reqno]{amsart}
\usepackage{amsmath}
\usepackage{amsfonts}
\usepackage{amssymb}
\usepackage{amsthm}
\usepackage{amscd}

\newcommand{\la}{\langle}
\newcommand{\ra}{\rangle}
\newcommand{\x}{\times}
\newcommand{\too}{\longrightarrow}
\newcommand{\lrightarrow}{\longrightarrow}

\newcommand{\SA}{{\mathcal{A}}}
\newcommand{\SC}{{\mathcal{C}}}
\newcommand{\SM}{{\mathcal{M}}}

\newcommand{\ZZ}{\mathbb{Z}}

\newcommand{\RR}{\mathbb{R}}

\DeclareMathOperator{\Id}{Id}

\numberwithin{equation}{section}

\theoremstyle{plain}
\newtheorem{proposition}{Proposition}[section]
\newtheorem{theorem}[proposition]{Theorem}
\newtheorem{lemma}[proposition]{Lemma}

\theoremstyle{definition}
\newtheorem{definition}[proposition]{Definition}
\newtheorem{examples}[proposition]{Examples}
\theoremstyle{remark}
\newtheorem{remark}[proposition]{Remark}

\begin{document}

\title[On formality of Sasakian manifolds]{On
formality of Sasakian manifolds}

\author[I. Biswas]{Indranil Biswas}
\address{School of Mathematics, Tata Institute of Fundamental
Research, Homi Bhabha Road, Bombay 400005, India}
\email{indranil@math.tifr.res.in}

\author[M. Fern\'andez]{Marisa Fern\'{a}ndez}
\address{Universidad del Pa\'{\i}s Vasco,
Facultad de Ciencia y Tecnolog\'{\i}a, Departamento de Matem\'aticas,
Apartado 644, 48080 Bilbao, Spain}
\email{marisa.fernandez@ehu.es}

\author[V. Mu\~{n}oz]{Vicente Mu\~{n}oz}
\address{Facultad de Ciencias Matem\'aticas, Universidad
Complutense de Madrid, Plaza de Ciencias
3, 28040 Madrid, Spain}
\email{vicente.munoz@mat.ucm.es}

\author[A. Tralle]{Aleksy Tralle}
\address{Department of Mathematics and Computer Science, University of Warmia
and Mazury, S\l\/oneczna 54, 10-710, Olsztyn, Poland}
\email{tralle@matman.uwm.edu.pl}



\begin{abstract}
We investigate some topological properties, in particular formality,
of compact Sasakian manifolds. Answering some
questions raised by Boyer and Galicki, we prove that all higher (than three) Massey 
products on any compact Sasakian manifold vanish. Hence, higher 
Massey products do obstruct Sasakian structures. Using this we produce a 
method of constructing simply connected $K$-contact non-Sasakian manifolds.

On the other hand, for every $n\, \geq\, 3$, we exhibit the first examples of simply
connected compact Sasakian manifolds of dimension $2n+1$ which 
are non-formal. They are non-formal because they have a non-zero triple Massey product.
We also prove that arithmetic lattices in some simple Lie groups cannot
be the fundamental group of a compact Sasakian manifold.
\end{abstract}

\maketitle

\section{Introduction}\label{sec:intro}

The present article deals with homotopical properties of Sasakian manifolds. Sasakian 
geometry has become an important and active subject, especially after the appearance of 
the fundamental treatise of Boyer and Galicki \cite{BG}. The Chapter 7 of this book 
contains an extended discussion on the topological problems in the theory of Sasakian, 
and, more generally, $K$-contact manifolds. In particular, there are several topological 
obstructions to the existence of the aforementioned structures on a compact manifold $M$ 
of dimension $2n+1$, for example:
\begin{itemize}
\item vanishing of the odd Stiefel-Whitney classes,

\item the inequality $1\,\leq\, cup(M)\,\leq\, 2n$ on the cup-length,

\item the evenness of the $p^{th}$ Betti number for $p$ odd with $1\, \leq\, p \, \leq\, n$,

\item the estimate on the number of closed integral curves of the Reeb vector field 
(there should be at least $n+1$),

\item some torsion obstructions in dimension 5 discovered by Koll\'ar.
\end{itemize}
Here we follow the above line of thinking. It is
well-known that the theory of
Sasakian manifolds is, in a sense, parallel to the theory of the K\"ahler manifolds. 
In fact, a Sasakian manifold is a 
Riemannian manifold $(M,g)$ 
such that $M\times {\mathbb{R}^+}$ equipped with the cone metric 
$h\,=\,t^2g+ {d}t^2$ is K\"ahler.
In particular, $M$ has odd dimension $2n+1$, where $n+1$ is the complex
dimension of the K\"ahler cone.
There is a deep theorem of Deligne, Griffiths, Morgan and Sullivan on the
rational homotopy type of K\"ahler manifolds \cite{DGMS}. In the same spirit, rational
homotopical properties of a manifold are related to the existence
of suitable geometric structures on the manifold \cite{FOT}.
Therefore, it is important to build a version of such theory for compact Sasakian
 manifolds. It seems that not much known in this direction, although some partial
results were obtained in \cite{Tievsky}.

In \cite[Chapter 7]{BG}, the authors pose the following problems.
\begin{enumerate}
\item Are there obstructions to the existence of Sasakian structures expressed in terms 
of Massey products?

\item There are obstructions to the existence of Sasakian structures expressed in terms 
of Massey products, which depend on basic cohomology classes of the related $K$-contact 
structure. Can one obtain a topological characterization of them?

\item Do there exist simply connected $K$-contact non-Sasakian manifolds (open problem 7.4.1)?

\item Which finitely presented groups can be realized as fundamental groups of compact 
Sasakian manifolds?
\end{enumerate}
The present paper deals with these problems. There are examples of non-simply connected
compact Sasakian manifolds which are non-formal because they have a non-zero triple Massey
product. Simply connected compact manifolds of dimension less than or equal to $6$ are
formal \cite{FM, N-Miller}. In Section \ref{non-formal-sasakian} we show that triple
Massey products do not obstruct Sasakian structures neither in the non-simply connected
case nor in the simply connected case. Indeed, in Theorem \ref{1-connected-sasak: non-formal}
we prove the following:

{\it For every $n\, \geq\, 3$ there exists a simply connected compact regular Sasakian 
manifold $M$, of dimension $2n+1$, with a non-zero triple Massey product}.

The example that we construct
is the total space of a non-trivial $3$-sphere bundle over 
$(S^2)^{n-1}=S^2 \times \stackrel{(n-1)}{\cdots} \times S^2$.

However, in Section \ref{k-cont-no-sasak}, we fully 
answer the question about Massey product obstructions
proving that {\it all higher order 
Massey products of Sasakian manifolds vanish} (Theorem 
\ref{prop:higher-massey} and Proposition \ref{prop:a-massey}). Using the 
latter we show a new method of constructing families of
compact $K$-contact manifolds with no Sasakian structures. More precisely, we have
the following result (Theorem \ref{thm:k-contact-ns}):

{\it
Let $M$ be a simply connected compact symplectic manifold of dimension $2k$
with an integral symplectic form $\omega$, and a non-zero 
quadruple Massey product. Then there exists a sphere bundle 
$S^{2m+1}\,\lrightarrow\, E\,\lrightarrow\, M$, 
with $m+1>k$, such that the total space $E$ is $K$-contact but has no Sasakian 
structure.}

Note that the existence of simply connected K-contact non-Sasakian manifolds was proved by Hajduk 
and the fourth author in \cite{HT} using the evenness of the third Betti number of a 
compact Sasakian manifold.

Finally, in Section \ref{fund-group}, we 
address the question $(4)$ on the fundamental groups of $K$-contact and 
Sasakian manifolds. We show that every finitely presented group occurs as 
the fundamental group of some compact $K$-contact manifold (Theorem 
\ref{thm:main}). In contrast, some arithmetic lattices in simple 
Lie groups cannot be realized as the fundamental group of some compact Sasakian 
manifold (Proposition \ref{prop:lattices1}).

\section{Formal manifolds and Massey products}\label{formal manifolds}

In this section some definitions and results about minimal models
and Massey products are reviewed \cite{DGMS, FHT}.

We work with {\em differential graded commutative algebras}, or DGAs,
over the field $\mathbb R$ of real numbers. The degree of an
element $a$ of a DGA is denoted by $|a|$. A DGA $(\mathcal{A},\,d)$ is {\it minimal\/}
if:
\begin{enumerate}
 \item $\mathcal{A}$ is free as an algebra, that is, $\mathcal{A}$ is the free
 algebra $\bigwedge V$ over a graded vector space $V\,=\,\bigoplus_i V^i$, and

 \item there is a collection of generators $\{a_\tau\}_{\tau\in I}$
indexed by some well ordered set $I$, such that
 $|a_\mu|\,\leq\, |a_\tau|$ if $\mu \,< \,\tau$ and each $d
 a_\tau$ is expressed in terms of preceding $a_\mu$, $\mu\,<\,\tau$.
 This implies that $da_\tau$ does not have a linear part.
 \end{enumerate}

In 
our context, the main example of DGA is the de Rham complex $(\Omega^*(M),\,d)$
of a differentiable manifold $M$, where $d$ is the exterior differential.

Given a differential graded commutative algebra $(\mathcal{A},\,d)$, we 
denote its cohomology by $H^*(\mathcal{A})$. The cohomology of a 
differential graded algebra $H^*(\mathcal{A})$ is naturally a DGA with the 
product structure inherited from that on $\mathcal{A}$ and with the differential
being identically zero. The DGA $(\mathcal{A},\,d)$ is {\it connected} if 
$H^0(\mathcal{A})\,=\,\RR$, and $\mathcal{A}$ is {\em $1$-connected\/} if, in 
addition, $H^1(\mathcal{A})\,=\,0$.

According to \cite[page 249]{DGMS}, an {\it elementary extension} 
of a differential graded commutative algebra $(\mathcal{A},\,d)$ is any 
DGA of the form 
$$(\mathcal{B}=\mathcal{A}\otimes\bigwedge V,\, d_{\mathcal{B}})\, ,$$ 
where 
$d_{\mathcal{B}}{|}_{\mathcal{A}}\,=\,d$, $d_{\mathcal{B}}(V)
\,\subset\,\mathcal{A}$ and $\bigwedge V$
is the free algebra over a finite dimensional vector space $V$ whose elements have all the same degree.

Morphisms between DGAs are required to preserve the degree and to commute 
with the differential. 

We shall say that $(\bigwedge V,\,d)$ is a {\it minimal model} of the
differential graded commutative algebra $(\mathcal{A},\,d)$ if $(\bigwedge V,\,d)$ is minimal and there
exists a morphism of differential graded algebras $$\rho\,\colon\,
{(\bigwedge V,\,d)}\longrightarrow {(\mathcal{A},\,d)}$$ inducing an isomorphism
$\rho^*\colon H^*(\bigwedge V)\stackrel{\sim}{\longrightarrow}
H^*(\mathcal{A})$ of cohomologies.

In~\cite{Halperin}, Halperin proved that any connected differential graded algebra
$(\mathcal{A},\,d)$ has a minimal model unique up to isomorphism. For $1$-connected
differential algebras, a similar result was proved earlier by Deligne, Griffiths,
Morgan and Sullivan~\cite{DGMS}.

A {\it minimal model\/} of a connected differentiable manifold $M$
is a minimal model $(\bigwedge V,\,d)$ for the de Rham complex
$(\Omega^*(M),\,d)$ of differential forms on $M$. If $M$ is a simply
connected manifold, then the dual of the real homotopy vector
space $\pi_i(M)\otimes \RR$ is isomorphic to $V^i$ for any $i$.
This relation also holds when $i\,>\,1$ and $M$ is nilpotent, that
is, the fundamental group $\pi_1(M)$ is nilpotent and its action
on $\pi_j(M)$ is nilpotent for all $j\,>\,1$ (see~\cite{DGMS}).

Recall that a minimal algebra $(\bigwedge V,\,d)$ is called
{\it formal} if there exists a
morphism of differential algebras $\psi\colon {(\bigwedge V,\,d)}\,\longrightarrow\,
(H^*(\bigwedge V),0)$ inducing the identity map on cohomology.
Also a differentiable manifold $M$ is called {\it formal\/} if its minimal model is
formal. Many examples of formal manifolds are known: spheres, projective
spaces, compact Lie groups, homogeneous spaces, flag manifolds,
and all compact K\"ahler manifolds.

The formality of a minimal algebra is characterized as follows.

\begin{proposition}[\cite{DGMS}]\label{prop:criterio1}
A minimal algebra $(\bigwedge V,\,d)$ is formal if and only if the space $V$
can be decomposed into a direct sum $V\,=\, C\oplus N$ with $d(C) \,=\, 0$,
and $d$ injective on $N$, such that every closed element in the ideal
$I(N)\, \subset\, \bigwedge V$ generated by $N$ is exact.
\end{proposition}

This characterization of formality can be weakened using the concept of
$s$-formality introduced in \cite{FM}.

\begin{definition}\label{def:primera}
A minimal algebra $(\bigwedge V,\,d)$ is $s$-formal
($s\,>\, 0$) if for each $i\,\leq\, s$
the space $V^i$ of generators of degree $i$ decomposes as a direct
sum $V^i\,=\,C^i\oplus N^i$, where the spaces $C^i$ and $N^i$ satisfy
the three following conditions:
\begin{enumerate}
\item $d(C^i) \,=\, 0$,

\item the differential map $d\,\colon\, N^i\,\longrightarrow\, \bigwedge V$ is
injective, and

\item any closed element in the ideal
$I_s\,=\,I(\bigoplus\limits_{i\leq s} N^i)$, generated by the space
$\bigoplus\limits_{i\leq s} N^i$ in the free algebra $\bigwedge
(\bigoplus\limits_{i\leq s} V^i)$, is exact in $\bigwedge V$.

\end{enumerate}
\end{definition}

A differentiable manifold $M$ is $s$-formal if its minimal model
is $s$-formal. Clearly, if $M$ is formal then $M$ is $s$-formal for all $s\,>\,0$.
The main result of \cite{FM} shows that sometimes the weaker
condition of $s$-formality implies formality.

\begin{theorem}[\cite{FM}]\label{fm2:criterio2}
Let $M$ be a connected and orientable compact differentiable
manifold of dimension $2n$ or $(2n-1)$. Then $M$ is formal if and
only if it is $(n-1)$-formal.
\end{theorem}

One can check that any simply connected compact manifold is $2$-formal. Therefore, 
Theorem \ref{fm2:criterio2} implies that any simply connected compact manifold of 
dimension not more than $6$ is formal.

In order to detect non-formality, instead of computing the minimal
model, which usually is a lengthy process, one can use Massey
products, which are known to be obstructions to formality. The simplest type
of Massey product is the triple (also known as ordinary) Massey
product, which we define next.

Let $(\mathcal{A},\,d)$ be a DGA (in particular, it can be the de Rham complex
of differential forms on a differentiable manifold). Suppose that there are
cohomology classes $[a_i]\,\in\, H^{p_i}(\mathcal{A})$, $p_i\,>\,0$,
$1\,\leq\, i\,\leq\, 3$, such that $a_1\cdot a_2$ and $a_2\cdot a_3$ are
exact. Write $a_1\cdot a_2\,=\,da_{1,2}$ and $a_2\cdot a_3\,=\,da_{2,3}$.
The {\it (triple) Massey product} of the classes $[a_i]$ is defined to be
$$
\langle [a_1],[a_2],[a_3] \rangle \,=\, 
[ a_1 \cdot a_{2,3}+(-1)^{p_{1}+1} a_{1,2}
\cdot a_3]
$$
$$
\in \, \frac{H^{p_{1}+p_{2}+ p_{3} -1}(\mathcal{A})}{[a_1]\cdot 
H^{p_{2}+ p_{3} -1}(\mathcal{A})+[a_3]\cdot H^{p_{1}+ p_{2} -1}(\mathcal{A})}\, .
$$
 
Now we move on to the definition of higher Massey products
(see~\cite{TO}). Given $$[a_{i}]\,\in\, H^*(\mathcal{A})\, ,\ \
1\,\leq\,i\,\leq\, t\, , \ \ t\,\geq\, 3\, ,$$
the Massey product $\la [a_{1}],[a_{2}],\cdots,[a_{t}]\ra$, is defined
if there are elements $a_{i,j}$ on $\mathcal{A}$, with $1\,\leq\, i\,\leq\, j\,\leq\,
t$ and $(i,j)\,\not= \,(1,t)$, such that
 \begin{equation}\label{eqn:gm}
 \begin{aligned}
 a_{i,i}&\,=\, a_i\, ,\\
 d\,a_{i,j}&\,=\, \sum\limits_{k=i}^{j-1} (-1)^{|a_{i,k}|}a_{i,k}\cdot
 a_{k+1,j}\, .
 \end{aligned}
 \end{equation}
Then the {\it Massey product} $\la [a_{1}],[a_{2}],\cdots,[a_{t}] \ra$ is the set of cohomology classes
 $$
 \la [a_{1}],[a_{2}],\cdots,[a_{t}] \ra
$$
$$
=\, \left\{
 \left[\sum\limits_{k=1}^{t-1} (-1)^{|a_{1,k}|}a_{1,k} \cdot
 a_{k+1,t}\right] \ \mid \ a_{i,j} \mbox{ as in (\ref{eqn:gm})}\right\}
 \,\subset\, H^{|a_{1}|+ \ldots +|a_{t}|
 -(t-2)}(\mathcal{A})\, .
 $$
We say that the Massey product is {\it zero} if
$$0\,\in\, \la [a_{1}],[a_{2}],\cdots,[a_{t}] \ra\, .$$ 

It should be mentioned that for $\la
a_{1},a_{2},\cdots,a_{t}\ra$ to be defined, it is necessary that
all the lower order Massey products $\la a_{1},\cdots,a_{i}\ra$ and
$\la a_{i+1},\cdots,a_{t}\ra$ with $2 \,<\, i \,<\, t-2$ are defined and
trivial. 

Massey products are related to formality by the
following well-known result.

\begin{theorem}[\cite{DGMS,TO}] \label{theo:Massey products}
A DGA which has a non-zero Massey product is not formal.
\end{theorem}

Another obstruction to the formality is given by the 
$a$-{\it Massey products} introduced in \cite{CFM}, that
generalize the triple Massey products. They have the advantage of
being simpler to compute than the higher order Massey
products. The $a$-Massey products are defined as follows.

Let $(\mathcal{A},\,d)$ be a DGA, and let
$a\, , b_1\, , \cdots\, , b_m \,\in\, \mathcal{A}$ be
closed elements such that the degree $|a|$ of $a$ is even and $a\cdot b_i$ is
exact for all $i$. Let $\xi_i$ be any element such that $d\xi_i \,=\, a
\cdot b_i$. Then, the {\it $m^{th}$ order $a$-Massey
product} of the $b_i$ is the subset
 $$
\la a; b_1,\cdots,b_m \ra
$$
$$
:=\, \left\{ \left[\sum_i (-1)^{|\xi_1|+\ldots + |\xi_{i-1}|}
{\xi_1}\cdot\ldots \cdot{\xi_{i-1}} \cdot b_i
\cdot \xi_{i+1}\cdot \ldots\cdot \xi_m\right] \, \mid\,
 d\xi_i \,=\, a\cdot b_i\right\}\,\subset\, H^{q}(\mathcal{A})\, ,
$$
where $q\,=\, (m-1)\deg a +(\sum_{i=1}^{m} \deg b_i)-m+1$.
We say that the $a$-Massey product is {\it zero} if $0\,\in\, \la
a; b_1,\cdots,b_n\ra$. 

\begin{lemma}[\cite{CFM}] \label{lem:cohomology a b}
The $a$-Massey product $\la a;b_1,\cdots,b_m\ra$ depends only on the
cohomology classes $[a]$, $\{[b_i]\}_{i=1}^m$, and not on the choices
of elements representing these cohomology classes.
\end{lemma}

\begin{theorem}[\cite{CFM}] \label{theo:amassey and formality}
A DGA which has a non-zero $a$-Massey product is not
formal.
\end{theorem}

The concept of formality is also defined for nilpotent CW-complexes, and all the above 
discussions can be extended to the nilpotent CW-complexes by using the DGA of piecewise 
polynomial differential forms ${\mathcal{A}_{PL}}(X)$ on a CW-complex $X$ (instead of 
using differential forms) \cite{FHT, GM}.

\section{Non-formal simply connected regular Sasakian manifolds}\label{non-formal-sasakian}

In this section we produce examples of simply connected compact
regular Sasakian manifolds, of dimension $\geq 7$, which are non-formal.
As mentioned before, Theorem \ref{fm2:criterio2}
gives that simply connected compact manifolds of dimension at most $6$
are all formal \cite{FM, N-Miller}.

First, we recall some definitions and results on Sasakian manifolds (see
\cite{BG} for more details).

Let $M$ be a $(2n + 1)$-dimensional manifold. An
{\em almost contact metric structure} on $M$ consists of
a quadruplet $(\eta\, , \xi\, , \phi\, ,g)$, where $\eta$ is a differential
$1$-form, $\xi$ is a nowhere vanishing vector field 
(known as the {\em Reeb} vector field), $\phi$ is a
$\SC^\infty$ section of ${\mathrm End}(TM)$ and $g$ is a Riemannian metric on $M$,
satisfying the following conditions
\begin{equation}\label{almostcontactmetric}
\eta(\xi) \,=\, 1, \quad \phi^2\,=\, - \Id + \xi \otimes \eta, \quad
g (\phi X\, , \phi Y)\,= \,g(X\, , Y) - \eta(X) \eta(Y)\, ,
\end{equation}
for all vector fields $X, Y$ on $M$. 
Thus, the kernel of $\eta$ defines a
codimension one distribution ${\mathcal D} \,=\, \ker (\eta)$, and there is
the orthogonal decomposition of the tangent bundle $TM$ of $M$
$$
TM = {\mathcal D} \oplus {\mathcal L}\, ,
$$
where ${\mathcal L}$ is the trivial line subbundle of $TM$ generated by $\xi$. Note
that conditions in \eqref{almostcontactmetric} imply that
\begin{equation} \label{alphaphi}
\phi(\xi)\,=\, 0\, ,\quad \eta\circ\phi\,=\,0\,.
\end{equation}

For an almost contact metric structure $(\eta\, , \xi\, , \phi\, ,g)$ on $M$, 
the fundamental 2-form $F$ on $M$ is defined by
$$
F (X, Y ) \,=\, g(\phi X, Y )\, ,
$$
where $X$ and $Y$ are vector fields on $M$. Hence, 
$$F(\phi X, \phi Y )\,=\,F(X, Y)\, ,$$ that is $F$ is compatible with
$\phi$, and $\eta\wedge F^n\not=\,0$ everywhere. 

An almost contact metric structure $(\eta\, , \xi\, , \phi\, ,g)$ on $M$ is said to
be a {\em contact metric} if
$$
g (\phi X, Y) \,=\, {d} \eta (X, Y)\,.
$$
In this case $\eta$ is a {\em contact form}, meaning
$$\eta\wedge ({d} \eta)^n\not=\,0$$ at every point of $M$. 
If $(\eta\, , \xi\, , \phi\, ,g)$ is a contact metric structure 
such that $\xi$ is a Killing vector field for $g$, meaning 
${\mathcal{L}}_{\xi} g\,=\,0$, where ${\mathcal{L}}_{\xi}$ denotes the Lie derivative, 
then $(\eta\, , \xi\, , \phi\, ,g)$ is called a
{\it K-contact} structure. A manifold with a K-contact structure is called a
{\it K-contact manifold}.

Just as in the case of an almost Hermitian structure, there is the notion of
integrability of an almost contact metric structure. More precisely, an almost contact
metric structure $(\eta\, , \xi\, , \phi\, , g)$ is called \emph{normal} if the Nijenhuis
tensor $N_{\phi}$ associated to the tensor field $\phi$, defined by
\begin{equation}\label{NPhi}
N_{\phi} (X, Y) \,:=\, {\phi}^2 [X, Y] + [\phi X, \phi Y] - 
\phi [ \phi X, Y] - \phi [X, \phi Y]\, ,
\end{equation}
satisfies the equation
$$
N_{\phi} \,=\,-{d}\eta\otimes \xi\, .
$$
This last equation is equivalent to the condition that the almost complex
structure $J$ on $M \times {\mathbb{R}}$ given by
\begin{equation} \label{complexproduct}
J \left( X,\, f\frac{\partial}{\partial t} \right)
\,=\, \left(\phi X - f \xi,\, \eta (X) \frac {\partial} {\partial t} \right)
\end{equation}
is integrable, where $f$ is a smooth function on $M \times {\mathbb{R}}$ and 
$t$ is the standard coordinate on ${\mathbb{R}}$ (see \cite{SH}). In other words, $\phi$
defines a complex structure on the kernel $\ker (\eta)$ compatible with $d\eta$.

A {\it Sasakian structure} is a normal contact metric structure, in other words, an 
almost contact metric structure $(\eta\, , \xi\, , \phi\, , g)$ such that
$$
N_ {\Phi} \,=\, -{d} \eta \otimes \xi\, , \quad {d} \eta \,=\, F\, .
$$
If $(\eta\, , \xi\, , \phi\, ,g)$ is a Sasakian structure on $M$, 
then $(M\, ,\eta\, , \xi\, , \phi\, ,g)$ is called a {\em Sasakian manifold}.

Riemannian manifolds with a Sasakian structure can also 
be characterized in terms of the Riemannian cone
over the manifold. A Riemannian manifold $(M,\, g)$ admits 
a compatible Sasakian structure if and only if $M\times {\mathbb{R}^+}$ 
equipped with the cone metric $h\,=\,t^2g+{d}t\otimes{d}t$ is K\"ahler \cite{BG}. 
Furthermore, in this case the Reeb vector field is Killing, and the covariant 
derivative of $\phi$ with respect to the Levi-Civita connection of $g$ is 
given by
\[ (\nabla_X \phi)(Y)\,=\,g(\xi, Y)X-g(X,Y)\xi\, , \]
where $X$ and $Y$ are vector fields on $M$.

In the study of Sasakian manifolds, the {\em basic cohomology} plays an important role. 
Let $(M,\eta\,, \xi\,, \phi\,,g)$ be a Sasakian manifold of dimension
$2n+1$. A differential form $\alpha$ on $M$ is called {\em basic} if
$$
\iota_{\xi}\alpha\,=\,0\,=\,\iota_{\xi}d\alpha\, ,
$$
where $\iota_{\xi}$ denotes the contraction of differential forms by $\xi$. We denote by 
$\Omega_{B}^{k}(M)$ the space of all basic $k$-forms on $M$. Clearly, the de Rham exterior 
differential $d$ of $M$ takes basic forms to basic forms. Denote by $d_B$ the restriction
of the de Rham differential $d$ to $\Omega_{B}^{*}(M)$. The cohomology of the 
differential complex $(\Omega_{B}^{*}(M),\, d_B)$ is called the {\em basic cohomology}.
These basic cohomology groups are denoted by $H_{B}^{*}(M)$. When $M$ is compact, the dimensions of 
the de Rham cohomology groups $H^{*}(M)$ and the basic cohomology groups $H_{B}^{*}(M)$
are related as follows:

\begin{theorem}[{\cite[Theorem 7.4.14]{BG}}]\label{rel:betti-basic-numbers}
Let $(M,\eta\, , \xi\, , \phi\, ,g)$ be a compact 
Sasakian manifold of dimension $2n+1$. Then, the Betti number $b_{r}(M)$ and the basic 
Betti number $b_{r}^{B}(M)$ are related by
$$
b_{r}^{B}(M)\, =\, b_{r-2}^{B}(M)+b_{r}(M)\, ,
$$
for $0\,\leq\, r\,\leq\, n$, In particular, if $r$ is odd
and $r\,\leq\, n$, then $b_{r}(M)\,=\,b_{r}^{B}(M)$.
\end{theorem}

A Sasakian structure on $M$ is called \textit{quasi-regular} if there is a positive
integer $\delta$ satisfying the condition
that each point of $M$ has a foliated coordinate chart
$(U\, ,t)$ with respect to $\xi$ (the coordinate $t$ is in the direction of $\xi$)
such that each leaf for $\xi$ passes through $U$ at most $\delta$ times. If
$\delta\,=\, 1$, then the Sasakian structure is called \textit{regular}. (See
\cite[p. 188]{BG}.)

A result proved in \cite{OV} says that if $M$ admits a Sasakian structure, then it 
also admits a quasi-regular Sasakian structure. It should be clarified that the word 
``fibration'' in \cite[Theorem 1.11]{OV} means ``orbi-bundle''.

If $N$ is a compact K\"ahler manifold whose K\"ahler form $\omega$ 
defines an integral cohomology class, then the total space of the circle bundle 
\begin{equation}\label{epi}
S^1 \,\hookrightarrow\, M \,\stackrel{\pi}{\longrightarrow}\, N
\end{equation}
with Euler class $[\omega]\,\in\, H^2(M,\,\mathbb{Z})$ is a regular Sasakian manifold with 
a contact form $\eta$ that satisfies the equation $d \eta \,=\, \pi^*(\omega)$, where 
$\pi$ is the projection in \eqref{epi}.

\subsection{A non-simply connected non-formal Sasakian manifold}

Recall from \cite{CFL} that the real Heisenberg group
$H^{2n+1}$ admits a homogeneous regular Sasakian structure with its standard 1-form
$\eta\,=\, dz-\sum_{i=1}^{n} y_{i}dx_{i}$. As a manifold
$H^{2n+1}$ is just ${\mathbb{R}^{2n+1}}$ which can be
realized in terms of $(n+2)\x (n+2)$ nilpotent matrices of the form
\begin{equation}\label{Heisenbergmatrix}
A= \left(
\begin{matrix}1 &a_1 &\cdots &a_n& c \\
0 &1 &0 &\cdots & b_1\\
\vdots && \ddots & \cdots & \vdots \\
0 & \cdots & 0&1 & b_n \\
0 &\cdots &0 & 0& 1
\end{matrix}
\right),
\end{equation}
where $a_i, b_i, c\,\in\, \RR$, $i\,=\,1\, ,\cdots\, , n$. Then a global system of
coordinates ${x_i\, , y_i\, , z}$ for $H^{2n+1}$ is defined by
$x_{i}(A)\,=\,a_{i}$, $y_{i}(A)\,=\,b_{i}$, $z(A)\,=\,c$. A standard
calculation shows that we have a basis for the left invariant $1$-forms
on $H^{2n+1}$ which consists of
$$
\{dx_{i}\, ,\, dy_{i}\, ,\, dz-\sum_{i=1}^{n} x_{i}dy_{i}\}\, .
$$

Consider the discrete subgroup $\Gamma$ of $H^{2n+1}$ defined
by the matrices of the form given in \eqref{Heisenbergmatrix} with integer entries.
The quotient manifold
 $$
 M\,:=\,\Gamma\backslash H^{2n+1}
 $$ 
is compact. The $1$-forms $dx_i$,
$dy_i$ and $dz-\sum_{i=1}^{n} x_{i}dy_{i}$ descend to $1$-forms $\alpha_i$, $\beta_i$
and $\gamma$ respectively on $M$. We note that
$\{\alpha_i\, , \beta_i\, , \gamma\}$ is a basis for the $1$-forms on $M$. 
Let $\{X_i\, , Y_i\, , Z\}$ be the basis of vector fields on $M$ that is dual to
the basis $\{\alpha_i\, , \beta_i\, , \gamma\}$.
Define the almost contact metric structure 
$(\eta\, , \xi\, , \phi\, , g)$ on $M$ by
 $$
 \eta\,=\, \gamma\, , \quad \xi\,=\,Z\, , \quad \phi(X_i)\,=\,Y_i\, ,\quad
\phi(Y_i)\,=\,-X_i
$$
$$
\phi( \xi)\,=\,0\, , \quad
 g\,=\,\gamma^2 + \sum_{i=1}^{n} ((\alpha_i)^2 + (\beta_i)^2)\, .
$$
Then one can check that $(\eta\,, \xi\,, \phi\, , g)$ is a regular Sasakian structure on $M$.
In fact, the manifold $M$ can be also defined as a
circle bundle over a torus $T^{2n}$. Moreover, $M$ is non-formal since it is
not $1$-formal in the sense of Definition \ref{def:primera}.

\subsection{A non-simply connected formal Sasakian manifold}

In order to construct an example of a formal compact non-simply connected
regular Sasakian manifold, we consider
the simply connected, solvable non-nilpotent Lie group $L^3$ of dimension $3$ consisting
of matrices of the form
\begin{equation}\label{Kahlersolvable}
A= \left(
\begin{matrix}\cos 2\pi c &\sin 2\pi c &0 &a \\
-\sin 2\pi c&\cos 2\pi c &0 &b\\
0 & 0&1 &c \\
0 &0 & 0& 1
\end{matrix}
\right)
\end{equation}
where $a, b, c\,\in\, {\mathbb{R}}$. Then a global system of
coordinates ${x\, , y\, , z}$ for $L^3$ is defined by
$x(A)\,=\,a$, $y(A)\,=\,b$, $z(A)\,=\,c$, and a standard
calculation shows that a basis for the right invariant $1$-forms
on $L^3$ consists of
$$
\{dx-2\pi y\,dz\, ,\, dy + 2\pi x \, dz\, ,\, dz\}\, .
$$

Notice that the solvable Lie group $L^3$ is not completely solvable.
Let $D$ be a discrete subgroup of $L^3$ such that the quotient space $L^3/ D$ is compact
(such a subgroup $D$ exists; see for example \cite{AGH} or \cite{TO}). The forms
$dx-2\pi y dz$, $dy + 2\pi x dz$ and $dz$ descend to 1-forms $\alpha$, $\beta$ and
$\gamma$ respectively on $L^3/ D$.

We define the product manifold $B^4\,=\,(L^3/ D) \times S^1$. 
Then, there are $1$-forms ${\alpha, \beta, \gamma,
\mu}$ on $B^4$ such that
$$
 d\alpha\,=\,-\beta\wedge\gamma, \quad d\beta\,=\,\alpha\wedge\gamma, \quad
d\gamma\,=\,d\mu\, =\,0\, ,
$$
and at each point of $B^4$, the cotangent vectors $\{\alpha\, , \beta\, , \gamma\, ,
\mu\}$ form a basis for the cotangent fiber. Moreover, 
using Hattori's theorem \cite{Hattori} it is straightforward to compute the real
cohomology groups of $B^4$, which are:
\begin{eqnarray*}
 H^0(B^4) &=& \la 1\ra, \\
 H^1(B^4) &=& \la [\gamma], [\mu]\ra,\\
 H^2(B^4) &=& \la [\alpha \wedge \beta], [\gamma\wedge \mu]\ra,\\
 H^3(B^4) &=& \la [\alpha \wedge \beta \wedge \gamma],
 [\alpha \wedge \beta \wedge \mu]\ra,\\
 H^4(B^4) &=& \la [\alpha \wedge \beta \wedge \gamma \wedge \mu]\ra.
\end{eqnarray*}

A K\"ahler structure $(g,\,\omega)$ on $B^4$ is given by the K\"ahler metric
$g\,=\,\alpha^2 +\beta^2 + \gamma^2 + \mu^2$ and the K\"ahler form
$$\omega\,=\, \alpha \wedge \beta + \gamma\wedge \mu\, .$$
According to \cite{Hasegawa}, the complex manifold $B^4$ is a finite quotient
of a compact complex torus.

We can suppose that the K\"ahler form $\omega$ on $B^4$ defines an integral cohomology class.
Therefore, the total space $M^5$ of the circle bundle over $B^4$ with Euler class
$[\omega]$ has a regular Sasakian structure. 

Clearly, $B^4$ is formal as it is a compact K\"ahler manifold.
In order to prove that $M^5$ is formal, we first observe that 
the minimal model of $B^4$ must be a differential graded algebra of the form
$({\SM},\,d)$, where ${\SM}\,=\,\bigwedge(a_1, a_2, b, c)$
is a free algebra such that the generators $a_i$ have
degree $1$, the generator $b$ has degree $2$
and the generator $c$ has degree $3$. The differential $d$ is given
by $da_i\,=\,db\,=\,0$ and $dc\,=\,b^2$.
A homomorphism $$\rho\,\colon\, {\SM} \,\longrightarrow\, \Omega(B^4)$$ that induces
an isomorphism on cohomology is defined by $\rho(a_1)\,=\,\gamma$,
$\rho(a_2)\,=\,\mu$, $\rho(b)\,=\,\alpha \wedge \beta$,
and $\rho(c)\,=\,0$.

Now, according to \cite{RS}, a (non-minimal) model of
$$S^1 \,\hookrightarrow \,M^5 \,\longrightarrow \,B^4$$ with Euler class
$[\omega]\,\in\, H^2(B^4,\,\mathbb{Z})$
is given by $\bigwedge(a_1, a_2, b, c) \otimes \bigwedge(a_3)$, where
$\bigwedge(a_1, a_2, b, c)$ is the above minimal model for $B^4$,
the generator $a_3$ has degree $1$ and its differential 
is given by $da_3\,=\,a_1 a_2 + b$. Then, the minimal model associated to this model
is $$({\widetilde \SM}\, , \,\widetilde d)\,=\,(\bigwedge(a_1, a_2, x)\, ,\,
\widetilde d)\, ,$$
where the generators $a_i$ have degree $1$ and the generator $x$ has degree
3 while the differential $\widetilde d$ is given
by $\widetilde da_i\,=\,\widetilde dx\,=\,0$. 
Therefore, we get $$C^1\,=\,\la a_1\, , a_2\ra\, ,\ \
N^1\,=\,0\,=\,C^2\,=\,N^2\, .$$ Now
Theorem \ref{fm2:criterio2} implies that $M^5$ is formal because it is $2$-formal.

\subsection{A simply connected non-formal Sasakian manifold}

The most basic example of a simply connected compact regular Sasakian manifold
is the odd-dimensional sphere $S^{2n+1}$ considered as the total space of the Hopf fibration
$S^{2n+1} \,\hookrightarrow \, {\mathbb{CP}^n}$. It is well-known that $S^{2n+1}$ is formal.
In the next theorem we give examples of non-formal simply
connected compact regular Sasakian manifolds.

\begin{theorem}\label{1-connected-sasak: non-formal}
For every $n\,\geq\, 3$, there exists a simply connected compact regular Sasakian manifold 
$M^{2n+1}$, of dimension $2n+1$, which is non-formal. More precisely, there is a 
non-trivial $3$-sphere bundle over $(S^2)^{n-1}$ which is a non-formal simply connected 
compact regular Sasakian manifold.
\end{theorem}

\begin{proof}
First take $n\,=\,3$.
We will determine a minimal model of the $7$-manifold $M^7$.
A minimal model
of the $3$-sphere $S^3$ is the differential algebra $(\bigwedge(z),\,d)$, where the generator
$z$ has degree $3$ and $dz =0$. A minimal model of 
$S^2 \times S^2$ is the differential 
algebra $(\bigwedge(a, b, x,y),\,d)$, where $a, b$
have degree $2$, while $x, y$ have degree $3$, and the differential $d$ is
given by the following: $da\,=\,db\,=\,0$,
$dx\,=\,a^2$ and $dy\,=\,b^2$. 
Therefore, a minimal model of the total space of a fiber bundle
 $$
 S^3 \,\hookrightarrow \,M^{7} \,\longrightarrow\, S^2\times S^2
 $$ 
is the differential algebra over the vector space $V$
generated by the elements $a, b$ of degree $2$ and $x,y,z$ of
degree $3$, and the differential $d$ is given by
 $$
 da\,=\,db\,=\,0\, , \quad dx\,=\,a^2\, , \quad dy\,=\,b^2\, , \quad dz\,=\,e\,a b\, ,
 $$
where $e\,[a b]\,\in\, H^4(S^2\x S^2,\,\ZZ)$ is the 
Euler class of the $S^3$-bundle with $e\,\in\,\mathbb{Z}$.

Let us assume that $e\,\neq\, 0$, so the $S^3$-bundle is non-trivial.
For $1\,\leq\, i\,\leq\, 4$, the subspace $V^i\, \subset\, V$ of degree $i$
decomposes as $V^i\,=\,C^i \oplus N^i$, where $C^1\,=\,N^1\,=\,0$,
$C^2\,=\,\la a, b \ra$, $N^2\,=\,0$, $C^3\,=\,0$, $N^3\,=\,\la x, y, z\ra$
and $C^4\,=\,N^4\,=\,0$.
Therefore, $(\bigwedge V, d)$ is $2$-formal because $N^1\,=\,N^2\,=\,0$. However the
minimal model 
$(\bigwedge V,\, d)$ is not $3$-formal because the element $\nu\,=\,a z -x b$ lies in
the ideal $N^{\leq 3} \bigwedge(V^{\leq 3})$ generated by $N^{\leq 3}$ in
$\bigwedge(V^{\leq 3})$, and $\nu$ is closed but non-exact
in $(\bigwedge V,\,d)$. This proves that
$M^{7}$ is non-formal because it is not $3$-formal.
In terms of Massey products, we have $a^2\,=\,dx$, 
$e\, a b\,=\,dz$,
which implies that the triple Massey product $\la a, a, b\ra$ is defined and
it is non-zero.

Now, to complete the proof for $n\,=\,3$ we need to show that 
$M^7$ is a regular Sasakian manifold
for suitable $e\neq 0$.
We will first show that $M$ can be considered as the circle bundle over
$S^2\times S^2\times S^2$ with Euler class
the K\"ahler form on $S^2\times S^2\times S^2$. Indeed, let $a_1,a_2,a_3$ be the generators of
the cohomology of each of the $S^2$-factors of $S^2\times S^2\times S^2$. Then
the K\"ahler form $\omega$ has cohomology class $[\omega]\,=\,a_1+a_2+a_3$. Consider
the principal $S^1$-bundle 
 $$
 S^1 \,\too\, N \,\too\, S^2\times S^2\times S^2
 $$ 
with first Chern class equal to $a_1+a_2+a_3$. Then the Gysin sequence gives that
 \begin{align*}
 H^0(N,\ZZ) &\,=\,H^7(N,\ZZ)=\ZZ, \\ 
 H^1(N,\ZZ) &\,=\,H^3(N,\ZZ)=H^6(N,\ZZ)=0, \\
 H^2(N,\ZZ) &\,= \,H^5(N,\ZZ)=\ZZ^2, \\
 H^4(N,\ZZ) &\,=\, \ZZ \la a_1a_2,a_1a_3,a_2a_3\ra/\la a_1a_2+a_1a_3,a_2a_1+
a_2a_3,a_3a_1+a_3a_2 \ra \,=\,\ZZ_{2}.
 \end{align*}

If we restrict to each $\{(x,y)\}\x S^2$, then this circle bundle has first Chern
class equal to $a_3$, the generator of $H^2(S^2,\, \ZZ)$. 
So this is the Hopf bundle $$S^1\,\too\, S^3 \,\too\, S^2\, .$$ Varying over all
$(x\, ,y)\, \in\, S^2\times S^2$, we have an $S^3$-bundle 
 $$
 S^3 \,\too\, N \,\too\, B\,=\,S^2\times S^2\, .
 $$
Let $e\, a_1a_2 \,\in\, H^4(B,\,\ZZ)$ be its Euler class, where $e\,\in\, \ZZ$. The Gysin
sequence gives
 \begin{align*}
 H^0(N,\ZZ) &\,=\,H^7(N,\ZZ)=\ZZ\, , \\ 
 H^1(N,\ZZ) &\,=\,H^3(N,\ZZ)=H^6(N,\ZZ)=0\, , \\
 H^2(N,\ZZ) &\,=\,H^5(N,\ZZ)=\ZZ^2\, , \\ 
 H^4(N,\ZZ) &\,=\,\ZZ_e\, .
 \end{align*}
Hence taking $e\,=\,2$, we have that $N\cong M$. Therefore, $M^7$ admits a regular Sasakian structure for $e=2$.

The case $n\,>\,3$ is similar; it is deduced as follows. Consider $B\,=\,S^2\times 
\stackrel{(n)}{\ldots} \times S^2$. Let $a_1\, ,\cdots\, ,a_{n}\,\in\, H^2(B)$ be the
cohomology classes of degree two given by each of the $S^2$-factors. Then the K\"ahler class
is given by $[\omega]\,=\,a_1+\ldots + a_{n}$. Consider the circle bundle 
$$
S^1\,\too \, N \,\too\, B
$$
with first Chern class equal to $[\omega]$. As above, this is an $S^3$-bundle over
$B'\,=\,S^2\times \stackrel{(n-1)}{\ldots} \times S^2$. The Euler class is 
$$
\sum_{1\leq i<j\leq n-1} e_{ij}\, a_ia_j \,\in\, H^4(B'),
$$
where $e_{ij}\,\in\, \ZZ$.
Restricting to each $S^2\x S^2$ embedded in the $i$-th and 
$j$-th factors, we see that $e_{ij}=2$ for every $i<j$, by the
computations in the case $n\,=\,3$.

The minimal model of such manifold $N$ is worked out as before. It is
$$
\bigwedge (a_1,\cdots, a_{n-1},x_1,\cdots, x_{n-1}, y)\, ,
$$ 
where $|a_i|\,=\,2$, $|x_i|\,=\,3$, $|y|\,=\,3$, $dx_i\,=\,a_i^2$, $dy\,=\,
2 \sum_{i<j} a_ia_j$.

We will show that $N^{2n+1}$ is not $3$-formal. For that,
firstly, $C^1\,=\,N^1\,=\,0$, $C^2\,=\,\la a_1, \cdots, a_{n-1} \ra$, 
$N^2\,=\,0$, $C^3\,=\,0$ and $N^3\,=\,\la x_1, \cdots, x_{n-1},y\ra$. 
As $H^{2n+1}(N)\,=\,\ZZ$, there is an element $\nu\,\in\, \bigwedge^{2n+1} V$
with $[\nu]\,\in\, H^{2n+1}(N)$ the generator. However, as $C$ is generated
by elements of even degree, it must be $\nu\,\in\, I(N)$ (actually,
$\nu\,=\,a_1\cdots a_{n-1}(y - \sum x_k a_l)$). So 
$N^{2n+1}$ is not $3$-formal. We now conclude that it
is non-formal because it is not $3$-formal.

Therefore, $N^{2n+1}$ is non-formal and has a regular Sasakian structure.
\end{proof}

\section{Simply connected compact $K$-contact manifolds with no
Sasakian structure}\label{k-cont-no-sasak}

In this section, we show that the higher Massey products
rule out the possibility of existence of Sasakian structures on a 
compact manifold. Moreover, using such an obstruction, we exhibit a new method
to construct simply connected $K$-contact non-Sasakian manifolds.

We will now recall the notions of symplectic and
contact fatness developed by Sternberg and Weinstein in the
symplectic setting, \cite{S}, \cite{W}, and by Lerman in the contact case
\cite{L1}, \cite{L2}. Let 
 $$
 G\,\lrightarrow\, P\,\lrightarrow\, B
 $$ 
be a principal $G$-bundle on $B$ equipped with a connection. Let $\theta$
and $\Theta$ respectively be the
connection one-form and the corresponding curvature 2-form on $P$. Both forms
have values in the Lie algebra
$\mathfrak{g}$ of the group $G$. Denote the natural pairing between
$\mathfrak{g}$ and its dual $\mathfrak{g}^*$ by
$\langle\,,\rangle$. By definition, a vector $u\,\in\,\mathfrak{g}^*$
is {\em fat} if the $2$-form
$$
(X,Y)\,\lrightarrow\, \langle\Theta(X,Y),u\rangle
$$
is nondegenerate for all horizontal vectors $X,Y$. Note that
if $u$ is fat, then each element of the coadjoint orbit of $u$ is fat.

Let $(M\, ,\eta)$ be a contact co-oriented manifold endowed with a contact action
of a Lie group $G$. Define a {\em contact moment map} by the formula
 $$
 \mu_{\eta}\,:\, M\,\lrightarrow\, \mathfrak{g}^*\, ,~ \,
 \langle \mu_{\eta}(x),X\rangle\,=\,\eta_x(X^*_x),
 $$
for any $x\,\in\, M$ and any $X\in\mathfrak{g}$, where $X^*$ denotes the
fundamental vector field on $M$ generated by $X\,\in\,\mathfrak{g}$ using the
action of $G$ on $M$. Note
that the moment map depends on the contact form. The theorem below is
due to Lerman.

\begin{theorem}[\cite{L2}]\label{thm:lerman-contact}
Let $(F,\, \eta)$ be a contact manifold equipped with an action of $G$ that
preserves $\eta$, and let $\nu$ be a contact moment map on $F$. Let
$$
G\,\lrightarrow\, P\,\lrightarrow\, M
$$
be a principal $G$-bundle endowed with a connection such that the
image $\nu(F)\,\subset\, \mathfrak{g}^*$ consists of fat vectors. Then there
exists a fiberwise contact structure on the total space of the
associated bundle
 $$
 F\,\lrightarrow \, P\times^G F\,\lrightarrow \, M\, .
 $$

If the fiber $(F,\, \eta)$ is $K$-contact, and $G$ preserves the
$K$-contact structure, then the total space $P\times^G F$
is also $K$-contact.
\end{theorem}

The second part of Theorem \ref{thm:lerman-contact} yields an explicit
construction of a fibered $K$-contact structure on a fiber bundle
and it will be our tool to prove Theorem \ref{thm:k-contact-ns}. 

Let $(M\, ,\omega)$ be a symplectic manifold such that the
cohomology class $[\omega]$ is integral. Consider the principal $S^1$-bundle
$$\pi \,: \,P\lrightarrow \, M$$ given
by the cohomology class $[\omega]\,\in\, H^2(M,\,\mathbb{Z})$.
Fibrations of this kind were first considered in \cite{BW} and are
called {\em Boothby-Wang fibrations}. By \cite{K}, the total space $P$
carries an $S^1$-invariant contact form $\theta$ such that $\theta$
is a connection form whose curvature is $\pi^*\omega$. This implies that
the moment map is constant and
nonzero. Moreover, by \cite{W}, a principal $S^1$-bundle is fat if
and only if it is a Boothby-Wang fibration. Therefore, we
have the following:

\begin{theorem}\label{thm:Boothby-Wang}
Let
$$S^1\,\lrightarrow\, P\,\lrightarrow \,M$$
be a Boothby-Wang fibration. Let $(F,\, \eta)$ be a contact manifold endowed with
an $S^1$-action preserving the contact form $\eta$. Then the associated
fiber bundle
$$
 F\,\lrightarrow\, P\times^{S^1}F\,\lrightarrow\, M
 $$
admits a fiberwise contact form. If $(F,\, \eta)$ is $K$-contact, then the
same is valid for the fiberwise contact structure on $P\times^{S^1}F$.
\end{theorem}

In order to prove that the higher order Massey products are zero for any compact 
Sasakian manifold, we proceed as follows. Let $(M,\eta\, , \xi\, , \phi\, ,g)$ be a 
compact Sasakian manifold. As in Section \ref{non-formal-sasakian}, we denote by 
$\Omega_{B}^{k}(M)$ the space of all basic $k$-forms on $M$, and by $d_B$ the restriction 
of the de Rham differential $d$ to $\Omega_{B}^{*}(M)$. As before, $H_{B}^{*}(M)$ 
denotes the basic cohomology of $M$. Let $[d\eta]_{B}\,\in \,H_{B}^{2}(M)$ be the basic 
cohomology class of $d\eta$, where $\eta$ is the above contact form (note that $d\eta$ 
is $d_B$-closed but not $d_B$-exact). Tievsky in \cite{Tievsky}, motivated by the 
approach in \cite{DGMS} and using the basic cohomology of $M$, determines a model for 
$M$, that is, a DGA with the same minimal model as the manifold $M$.

\begin{theorem}[\cite{Tievsky}] \label{Tievsky model}
Let $(M,\eta\, , \xi\, , \phi\, ,g)$ be a compact Sasakian manifold. Then a model for $M$ is
given by the DGA $(H_{B}^{*}(M)\otimes\bigwedge(x),\, D)$, where $|x|\,=\,1$,
$D(H_{B}^{*}(M))\,=\,0$ and $Dx\,=\,[d\eta]_{B}$. Therefore,
$(H_{B}^{*}(M)\otimes\bigwedge(x), \,D)$ is an elementary 
extension of the DGA $(H_{B}^{*}(M),\, 0)$ (the differential is zero).
\end{theorem}

\begin{theorem}\label{prop:higher-massey}
Let $M$ be a compact Sasakian manifold. Then, all the higher order
Massey products for $M$ are zero. 
\end{theorem}

Before proving Theorem \ref{prop:higher-massey}, we recall the definition of the {\it hard
Lefschetz property} and prove Proposition \ref{prop:Leschetz-higher-massey} below.

Let $(\mathcal{A}\,=\,\bigoplus_{i=0}^{2n}A^i,0)$ be a graded differential algebra with a
non-zero element $\omega\,\in\, A^2$. For every $0\,\leq\, k \,\leq\, n$,
define the Lefschetz map
$$
L_\omega\,:\, A^{n-k} \,\longrightarrow\, A^{n+k}\, ,\ \
\beta\,\longmapsto\, \beta\cdot\omega^{n-k}\, .$$
We say that $\mathcal{A}$ satisfies the hard Lefschetz property if $L_\omega$ is an
isomorphism for every $0 \,\leq\, k \,\leq \,n$. 

\begin{proposition}\label{prop:Leschetz-higher-massey}
Let $(\mathcal{A}\,=\,\bigoplus_{i=0}^{2n} A^i,\, 0)$ be a differential graded commutative
algebra, and let $\omega\,\in\, A^2$ be a nondegenerate element, such that the hard
Lefschetz property with respect to $\omega$ holds. Consider the elementary extension
$(\mathcal{A}\otimes\bigwedge (y), \,d)$ of $(\mathcal{A},\, 0)$, where $\deg(y)\,=\,1$ and
$dy\,=\,\omega$. Then the higher order Massey products $\langle a_1,\cdots,
 a_m\rangle$, where $a_i\,\in\, H^*(\mathcal{A}\otimes \bigwedge (y),\, d)$ and
$m\,\geq\, 4$, all vanish.
\end{proposition}

 \begin{proof} 
Since $\mathcal{A}\,=\,\bigoplus_{i=0}^{2n} A^i$ satisfies the hard
Lefschetz property with respect to $\omega$, the map
 $$
 L_{\omega}^{n-p} \,\colon\, A^{p} \,\longrightarrow\, A^{2n-p}\, ,\ \
\alpha\,\longmapsto\, \alpha\cdot\omega^{n-p}
$$
is an isomorphism for all $p \,\leq\, n-1$. 
Let $a\,\in\, H^k(\mathcal{A}\otimes \bigwedge (y),d)$. Then
$$a\,=\,[\alpha+\beta\cdot y]\, ,$$ where
$\alpha, \beta\,\in\, \mathcal{A}$ and $d(\alpha+\beta\cdot y)\,=\,\overline{\beta}\cdot\omega
\,=\,0$. Here, $\overline{\beta}$ stands for $(-1)^{\deg \beta}\beta$, and so on.
If $\deg a\,\leq\, n$, then $\deg\beta\,\leq\, n-1$.
 By the injectivity of the map $L_\omega\,\colon\, A^{p} \,\longrightarrow \, A^{p+2}$ for
 $p\,\leq\, n-1$, we have $\beta\,=\,0$. Thus we have $a\,=\,[\alpha]$.
However, if $\deg a\,\geq\, n+1$, then $\deg \alpha\,\geq \, n+1$.
By the surjectivity of
$$L_\omega\,\colon\, A^{q-2} \,\longrightarrow \, A^{q}$$ for $q \geq n+1$,
 we have that $\alpha\,=\,\gamma\cdot\omega\,=\,d(\overline{\gamma}\cdot y)$. Hence
$a\,=\,[\beta\cdot y]$.

Let us consider a higher order Massey product $$\langle a_1,\cdots, a_m\rangle\, ,$$ 
where $a_i\,\in\, H^*(\mathcal{A}\otimes \bigwedge (y),\,d)$ and $m\,\geq\, 4$.
 The degree of the Massey product $\langle a_1\, ,\cdots\, , a_m\rangle$ is
$(\sum_{i=1}^{m} \deg a_i)-m+2$. As 
 $$(\sum_{i=1}^{m} \deg a_i)-m+2\,\leq\, 2n+1\, ,$$
it follows that there is at most one $a_i$ with $\deg a_i\,\geq\, n+1$.

Suppose that the higher Massey product $\langle a_1,\cdots, a_m\rangle$ is defined, and all 
$\deg a_i\,\leq\, n$. Set $a_i\,=\,[\alpha_i]$. According to Section \ref{formal 
manifolds}, any element $b\,\in\, \langle a_1,\cdots, a_m\rangle$ is a cohomology class
in $H^*(\mathcal{A}\otimes \bigwedge (y),\, d)$, which is obtained by taking 
$$\gamma_{ij}\,=\,\alpha_{ij}+\beta_{ij}\cdot y\in \mathcal{A}\otimes \bigwedge (y)\, ,$$
where $1\,\leq \,i\,\leq\, j \,\leq\, m$ and $(i, j)\not\,=\,(1,m)$, such that 
 $\gamma_{ii}\,=\,\alpha_{ii}\,=\,\alpha_i$ and
\begin{equation} \label{eqn:1}
d\gamma_{ij}\,=\,\sum_{i\leq k\leq j-1} \overline{\gamma_{ik}}\cdot\gamma_{k+1,j}\, .
\end{equation}
(Here it is used that we can fix the representatives $\alpha_i$ of the cohomology
classes $a_i$ \cite{BT}.)
Then, we have
 $$
b\,=\,\Big[\sum_{1\leq k\leq m-1} \overline{\gamma_{1k}}\cdot\gamma_{k+1,m}\Big]\, .
 $$
Note that the equations in \eqref{eqn:1} are equivalent to the following
\begin{equation} \label{eqn:2}
\overline{\beta_{ij}}\cdot\omega\,=\, \sum_{i\leq k\leq j-1} \overline{\alpha_{ik}}\cdot\alpha_{k+1,j} \quad \text{and} \quad 
0\,=\,\sum_{i\leq k\leq j-1} (\overline{\alpha_{ik}}\cdot\beta_{k+1,j}-\overline{\beta_{ik}}\cdot\overline{\alpha_{k+1,j}})\, .
\end{equation}

Now, we choose 
$$\widetilde\gamma_{ij}\,=\,\widetilde\alpha_{ij}+\widetilde\beta_{ij}\cdot\, y\,\in\, 
\mathcal{A}\otimes \bigwedge (y)\, ,\ \ 1\,\leq\, i\,\leq\, j \,\leq\, m\, ,
\ \ (i, j)\,\not=\, \,(1,m)\, ,$$ as follows:
$$
 \widetilde\gamma_{ii}\,=\,\widetilde\alpha_{ii}\,=\,\alpha_i\, , \ \
\widetilde\gamma_{ij}\,=\,\widetilde\beta_{ij}\cdot y\,=\,\beta_{ij}\cdot y \ \ 
\text{for} \ \ j\,=\,i+1\, , \ \ \text{and} \ \ \widetilde\gamma_{ij}\,=\,0 \ \  \text{for}
\ \ j\,\geq\, i+2\, .
 $$ 
 Then, using \eqref{eqn:2}, one can check that
 $$
 \overline{\widetilde\beta_{ij}}\cdot\omega\,=\, \sum_{i\leq k\leq j-1} \overline{\widetilde\alpha_{ik}}\cdot\widetilde\alpha_{k+1,j} \qquad 
 \text{and} \qquad 
 0\,=\,\sum_{i\leq k\leq j-1}(\overline{\widetilde\alpha_{ik}}\cdot\widetilde\beta_{k+1,j}-\overline{\widetilde\beta_{ik}}\cdot\overline{\widetilde\alpha_{k+1,j}}).
 $$
This means that the elements $\widetilde\gamma_{ij}$ satisfy
$$d{\widetilde\gamma}_{ij}\,=\,\sum_{i\leq k\leq j-1} 
\overline{{\widetilde\gamma}_{ik}}\cdot{\widetilde\gamma}_{k+1,j}\, ,$$
and so they define an element $\widetilde b\,\in\,\langle a_1,\cdots, a_m\rangle$ given by 
$$
{\widetilde b}\,=\, [\sum_{1\leq k \leq m-1}
 \overline{{\widetilde\gamma}_{1k}}\cdot\widetilde\gamma_{k+1,m}]\, ,
 $$
which is the zero cohomology class because every term of 
$\sum_{1\leq k \leq m-1} \overline{{\widetilde\gamma}_{1k}}\cdot\widetilde\gamma_{k+1,m}$ vanishes.
 
To complete the proof of the proposition, suppose that the higher Massey product $\langle a_1,\cdots, a_m\rangle$
is defined but there is one cohomology class $a_t$ with
$1\,\leq\, t \,\leq \,m$ and $\deg a_t\,\geq\, n+1$. Then
 $a_i\,=\,[\alpha_i]$ for $i\,\neq \,t$,
 and $a_t\,=\,[\beta_t\cdot y]$. So a cohomology class $b\,\in\,
\langle a_1\, ,\cdots\, , a_m\rangle$ is given by taking 
 $\gamma_{ij}\,=\,\alpha_{ij}+\beta_{ij}\cdot y$, where
 $1\,\leq\, i\,\leq\, j \,\leq\, m$ and $(i, j)\,\neq \,(1, m)$, such that 
 $$
 \gamma_{ii}\,=\,\alpha_{ii}\,=\,\alpha_i, \quad \text{for} \quad i\neq t, \qquad \gamma_{tt}\,=\,\beta_{tt}\cdot y\,=\,\beta_t\cdot y\, , \quad 
 \text{and}  \quad d\gamma_{ij}\,=\,\sum_{i\leq k\leq j-1}
\overline{\gamma_{ik}}\cdot\gamma_{k+1,j}\, .
 $$
 Thus, we have
$b\,=\,[\sum_{1\leq k\leq m-1} \overline{\gamma_{1k}}\cdot\gamma_{k+1,m}]$. Again,
$$d\gamma_{ij}\,=\,\sum_{i\leq k\leq j-1} \overline{\gamma_{ik}}\cdot\gamma_{k+1,j}$$ are
equivalent to
$$\overline{\beta_{ij}}\cdot\omega\,=\,\sum_{i\leq k\leq j-1}
\overline{\alpha_{ik}}\cdot\alpha_{k+1,j}\ \ \text{ and }\ \ 
0\,=\,\sum_{i\leq k\leq j-1} (\overline{\alpha_{ik}}\cdot\beta_{k+1,j}-
\overline{\beta_{ik}}\cdot\overline{\alpha_{k+1,j}})\, .$$
In particular, 
 $$
 \beta_{t-1,t}\cdot\omega\,=\,\beta_{t,t+1}\cdot\omega\,=\,0 \qquad \text{and}   \qquad  
\alpha_{t-1}\cdot \beta_t \,=\, 0\, ,
 $$
since $\alpha_{t}=0\,=\,\beta_{t-1}$. 
 
Now, we consider the elements $\widetilde\gamma_{ij}=\widetilde\alpha_{ij}+\widetilde\beta_{ij}\cdot\, y\in \mathcal{A}\otimes \bigwedge (y)$, $1\leq i\leq j \leq m$ and $(i, j)\not=(1,m)$, given by
 \begin{equation*}
\begin{aligned}
\widetilde\gamma_{ii}&=\widetilde\alpha_{ii}=\alpha_i,  \quad \text{for} \quad  i\neq t, \quad \qquad \widetilde\gamma_{tt}=\widetilde\beta_{tt}\cdot y=\beta_{t}\cdot y, \\
\widetilde\gamma_{ij}&=\widetilde\beta_{ij}\cdot\, y=\beta_{ij}\cdot\, y, \quad \text{for} \quad j=i+1\quad  \text{and}\quad i\neq t-1, t, \\
\widetilde\gamma_{t-1,t}&=\widetilde\gamma_{t,t+1}=0, \\
\widetilde\gamma_{i,j}&=0, \quad \text{for} \quad j\geq i+2.
  \end{aligned}
 \end{equation*}
Then,
 $\overline{\widetilde\beta_{ij}}\cdot\omega= \sum_{i \leq k \leq j-1}
 \overline{\widetilde\alpha_{ik}}\cdot\widetilde\alpha_{k+1,j}$ and
 $0=\sum_{i\leq k \leq j-1}
 (\overline{\widetilde\alpha_{ik}}\cdot\widetilde\beta_{k+1,j}-\overline{\widetilde\beta_{ik}}\cdot\overline{\widetilde\alpha_{k+1,j}})$.
 This defines the element 
 $$
 \widetilde b=[\sum_{1 \leq k\leq m-1}
 \overline{\widetilde\gamma_{1k}}\cdot\widetilde\gamma_{k+1,m}],
 $$
which is the zero cohomology class because each term of 
$\sum_{1 \leq k\leq m-1}\overline{\widetilde\gamma_{1k}}\cdot\widetilde\gamma_{k+1,m}$ vanishes.
\end{proof}

\begin{remark}\label{ref-rem}
It should be mentioned that Proposition \ref{prop:Leschetz-higher-massey} is false without 
the assumption of the hard Lefschetz property for $\mathcal{A}$. This was realized after
the referee produced a counterexample.
\end{remark}

\bigskip

\noindent{\bf Proof of Theorem~\ref{prop:higher-massey}~:} 
We know that Massey products on a manifold $M$
can be computed by using any model for $M$. Now, let $(M,\eta\, , \xi\, , \phi\, ,g)$ be 
a compact Sasakian manifold of dimension $2n+1$.
Put $\mathcal{A}\,=\,H_{B}^{*}(M)$ the basic cohomology of $M$. From Theorem \ref{Tievsky model} we know that
$(\mathcal{A}\otimes\bigwedge(x),\, D)$, with $|x|\,=\,1$, $D(\mathcal{A})\,=\,0$ and $Dx\,=\,[d\eta]_{B}$,
is a model for $M$.

Since $(M,\eta\, , \xi\, , \phi\, ,g)$ is a Sasakian manifold, 
there is the orthogonal decomposition of the tangent bundle $TM$ of $M$
$$
TM = {\mathcal D} \oplus {\mathcal L}\, ,
$$
where ${\mathcal L}$ is the trivial line subbundle generated by $\xi$, and the transversal
foliation ${\mathcal D}$ is K\"ahler with respect to $\omega\,=\,[d\eta]_{B}\in H_{B}^{2}(M)$.
Then, we can apply El Kacimi's basic hard Lefschetz Theorem for transversally K\"ahler foliations $\mathcal{F}$
on a compact manifold $N$ such that $H_{B}^{cod\,\mathcal{F}}(N)\neq\,0$. (See \cite{El Kacimi-Alaoui}
for a more complete statement of the theorem.)
This implies that $(\mathcal{A}, 0)$ satisfies the hard Lefschetz property with respect to $\omega$ and so, by 
Proposition \ref{prop:Leschetz-higher-massey}, all higher Massey products 
of the model $(\mathcal{A}\otimes\bigwedge(x), D)$ are zero.
\hfill$\square$

\begin{proposition}\label{prop:Leschetz-a-massey} 
In the set-up of Proposition \ref{prop:Leschetz-higher-massey}, all $m^{th}$ order
$a$-Massey products 
$\la a; b_1,\cdots,b_m \ra$ of $(\mathcal{A}\otimes \bigwedge (y), d)$, where $m\,\geq\, 3$,
are zero.
\end{proposition}

\begin{proof} 
Suppose that the $a$-Massey product $\la a; b_1,\cdots,b_m \ra$ is defined, where $a$ and $b_i$ are closed elements 
in $(\mathcal{A}\otimes \bigwedge (y), d)$, $\deg a =$ even and $m\geq 3$. Then, the degree of $\la a; b_1,\cdots,b_m \ra$
 is 
 $$
 \deg \la a; b_1,\cdots,b_m \ra\,=\, (m-1)\deg a +(\sum_{i=1}^{m} \deg b_i)-m+1\leq 2n+1.
 $$
This implies $\deg a \leq n$ and there is at most one $b_i$ with $\deg b_i\geq n+1$.

By Lemma \ref{lem:cohomology a b}, 
the $a$-Massey product $\la a;b_1,\cdots,b_m\ra$ only depends on the
cohomology classes $[a]$, $[b_i]$ and not on the particular
elements representing these classes.

If all $\deg b_i\leq n$, then the Lefschetz property for $(\mathcal{A}, \, \omega)$ implies
that not only $a\,\in\, \mathcal{A}$ but also $b_i\,\in\, \mathcal{A}$,
for all $1\,\leq\, i \,\leq\, m$. 
According to section \ref{formal manifolds}, an element $c\in \la a; b_1,\cdots,b_m \ra$ is
a cohomology class 
in $H^*(\mathcal{A}\otimes \bigwedge (y),d)$, which is given by taking 
 $\xi_i\,=\,\mu_{i}+\nu_{i}\cdot y\in \mathcal{A}\otimes \bigwedge (y)$, where $1\,\leq\, i
\,\leq\, m$, such that 
 $d\xi_i \,=\, a\cdot b_i$.
 Then, we have
 $$
c\,=\,\Big[\sum_i (-1)^{|\xi_1|+\ldots + |\xi_{i-1}|}\,
{\xi_1}\cdot\ldots \cdot{\xi_{i-1}} \cdot b_i
\cdot \xi_{i+1}\cdot \cdots\cdot \xi_m\Big]\, .
 $$

Now, from the elements $\xi_i$, we choose $\widetilde\xi_i\,=\,\widetilde\mu_{i}+
\widetilde\nu_{i}\cdot\, y\in \mathcal{A}\otimes \bigwedge (y)$, $1\,\leq\, i \,\leq\, m$,
given by
$$
\widetilde\xi_i\,=\,\widetilde\nu_{i}\cdot\, y\,=\,\nu_{i}\cdot\, y\, .
 $$ 
Since $\mu_i\in\mathcal{A}$, we have ${d}\mu_i\,=\,0$. 
So ${d}\widetilde{\xi}_i\,=\,d(\nu_i\cdot y)\,=\,{d}\xi_i\,=\,a\cdot b_i$.
This means that the elements $\widetilde{\xi}_i$ define the element
$$
 \widetilde c\,=\, \Big[\sum_i (-1)^{|\xi_1|+\ldots + |\xi_{i-1}|}\,
{\widetilde\xi_1}\cdot\ldots \cdot{\widetilde\xi_{i-1}} \cdot b_i
\cdot \widetilde\xi_{i+1}\cdot \ldots\cdot \widetilde\xi_m\Big]\,=\,0,
 $$
because each term of 
 $\sum_i (-1)^{|\xi_1|+\ldots + |\xi_{i-1}|}\,{\widetilde\xi_1}\cdot\ldots \cdot{\widetilde\xi_{i-1}} \cdot b_i
\cdot {\widetilde\xi_{i+1}}\cdot \ldots\cdot {\widetilde\xi_m}$ vanishes.

Suppose that $\la a;b_1,\cdots,b_m\ra$ is defined
but there is one $b_t$ $(1\leq t \leq m)$ with $\deg b_t\geq n+1$. Thus,
 $a, b_i\in {\mathcal A}$, for $i\neq t$,
 and $b_t\,=\,\beta_t\cdot y$. So a cohomology class $c\in \la a;b_1,\cdots,b_m\ra$ is given by taking 
 $\xi_i\,=\,\mu_{i}+\nu_{i}\cdot y\in \mathcal{A}\otimes \bigwedge (y)$, where $1\leq i \leq m$, such that 
 $d\xi_i \,=\, a\cdot b_i$. Then, 
 $$
c=\Big[\sum_i (-1)^{|\xi_1|+\ldots + |\xi_{i-1}|}\,
{\xi_1}\cdot\ldots \cdot{\xi_{i-1}} \cdot b_i
\cdot \xi_{i+1}\cdot \ldots\cdot \xi_m\Big] .
 $$
We choose $\widetilde\xi_i\,=\,\widetilde\mu_{i}+\widetilde\nu_{i}\cdot\, y\in \mathcal{A}\otimes \bigwedge (y)$, $1\leq i \leq m$,
as follows
$$
\widetilde\xi_i\,=\,\widetilde\nu_{i}\cdot\, y\,=\,\nu_{i}\cdot\, y, \quad \text{for} \quad i\neq t, \quad \qquad \widetilde\xi_t\,=\,0.
 $$ 
 Hence, $\widetilde\xi_i$ satisfy $d\widetilde\xi_i \,=\, a\cdot b_i$. For $i \,\neq\, t$, we have 
 $$
 d\xi_{i}\,=\,d\nu_i\cdot y\,=\,a\cdot b_i\, ,
 $$ 
and for $i\, =\, t$ we have $d\widetilde\xi_t\,=\,0$, and 
 $$
 d\xi_t \,=\, a\cdot b_t \, =\, a\cdot \beta_t\cdot y\,.
$$
As the image of $d$ is contained in $\mathcal{A}$, it follows that $a\cdot \beta_t\, =\,0$ and hence
$a\cdot b_t\, =\, 0\,=\,d\widetilde\xi_t$. 
Therefore, the elements $\widetilde\xi_i$ define the cohomology class $\widetilde c
\,\in\,\la a; b_1,\cdots,b_m \ra$ given by $ \widetilde c\,
 =\, \Big[\sum_i (-1)^{|\xi_1|+\ldots + |\xi_{i-1}|}\,
{\widetilde\xi_1}\cdot\ldots \cdot{\widetilde\xi_{i-1}} \cdot b_i
\cdot \widetilde\xi_{i+1}\cdot \ldots\cdot \widetilde\xi_m\Big]\,=\,0$.
\end{proof}

\begin{proposition}\label{prop:a-massey} 
All $a$-Massey products of a compact Sasakian manifold are zero.
\end{proposition}

\begin{proof}
Let $(M,\eta\, , \xi\, , \phi\, ,g)$ be a compact Sasakian manifold.
As we noticed in the proof of Theorem \ref{prop:higher-massey}, the differential graded commutative algebra 
$(\mathcal{A}\,=\,H_{B}^{*}(M), 0)$, where $\omega\,=\,[d\eta]_{B}\in H_{B}^{2}(M)$, satisfies the hard Lefschetz property with respect to
$\omega$. Then, Proposition \ref{prop:Leschetz-a-massey} implies that all $a$ Massey products 
of the model $(\mathcal{A}\otimes\bigwedge(x), D)$ for $M$ are zero.
\end{proof}

\begin{theorem}\label{thm:k-contact-ns}
Let $M$ be a simply connected compact symplectic manifold of dimension $2k$
with an integral symplectic form $\omega$. Assume that the
quadruple Massey product in $H^*(M)$ is non-zero. There exists a sphere bundle 
 $$
 S^{2m+1}\,\lrightarrow\, E\,\lrightarrow\, M\, ,
 $$
for $m+1\,>\,k$, such that the total space $E$ is $K$-contact, but $E$ does not
admit any Sasakian structure.
\end{theorem} 

\begin{proof}
Let
$$S^1\,\lrightarrow\, P\,\lrightarrow\, M$$
be the principal $S^1$-bundle corresponding to $[\omega]\,\in\, H^2(M,\,\mathbb{Z})$.
Choose a unitary representation of $S^1$ in $\mathbb{C}^{m+1}$
whose all weights are positive. Consider the $S^{2m+1}$-bundle
$$S^{2m+1}\, \lrightarrow\, E\,:=\, P\times^{S^1}S^{2m+1}\,\lrightarrow\, M$$
associated to the principal $S^1$-bundle $P$ for this unitary representation.
By Theorem \ref{thm:Boothby-Wang} applied to $S^{2m+1}$, we obtain a fiberwise
$K$-contact structure on the total space $E$ (see also \cite{HT}).
Clearly, there is an algebraic model of $E$ of the form
$$(\mathcal{A}_{PL}(M),\,d)\,\lrightarrow \,(\mathcal{A}_{PL}\otimes\bigwedge (z),\,D)
\,\lrightarrow\, (\bigwedge (z),0)\, ,$$
where $(\bigwedge (z),0)$ is the minimal model of $S^{2m+1}$ with one odd generator
$z$ of degree $2m+1$. Note that by the degree reasons, $D(z)\,=\,0$. Indeed,
$D(z)$ represents
a cohomology class in $H^{2m+2}(M)$ (it is the Euler class of the corresponding
vector bundle). Since $2m+2\,>\,2k\,=\,\dim M$, this class has to be zero.
But this means that $\SA_{PL}(E)$ must be weakly equivalent $\SA_{PL}(M\times S^{2m+1})$
(see \cite[p. 202, Example 4]{FHT}). It now follows that 
$$(\mathcal{A}_{PL}(M)\otimes \bigwedge (z),\,D)\,\simeq\, (\mathcal{A}_{PL}(M),\,d)
\otimes (\bigwedge(z),0)\, ,$$
and the latter is a model of $E$. Assume now that $E$ is Sasakian. By Theorem
\ref{prop:higher-massey}, all higher order Massey products in $H^*(E)$ must be
zero. But this contradicts the assumption in the theorem because
$(\mathcal{A}_{PL}(M),\,d) \otimes (\bigwedge(z),0)$ is a model of $E$. Therefore,
$E$ is not Sasakian.
\end{proof}

\begin{remark}
Theorem \ref{thm:k-contact-ns} can also be proved using the
differential algebra $(\Omega^*(M),\,d)$ of the de Rham forms instead of
$\mathcal{A}_{PL}(M)$. We prefer the latter because it enables us to use citation
of \cite{FHT} directly.
\end{remark}

\begin{examples}\label{l.ex}
\mbox{}
\begin{enumerate}
\item It is proved in \cite{BT} that there exist simply connected compact 
symplectic manifolds with non-zero quadruple Massey products. Therefore, any 
such manifold $M$ is a base of some sphere bundle which is $K$-contact but 
non-Sasakian.
 
\item There exists an $8$-dimensional simply connected 
compact symplectic manifold 
$M$ with a non-zero triple $a$-Massey product. It was constructed in \cite{FM1}. 
One can easily modify Theorem \ref{thm:k-contact-ns}, assuming that the base $M$ 
is symplectic, simply connected and possesses a non-zero triple $a$-Massey 
product. In this way one obtains a $17$-dimensional $K$-contact non-Sasakian simply 
connected compact manifold. More such manifolds can be constructed using results of
\cite{CFM}, since the property of having non-zero $a$-Massey products is 
inherited by symplectic blow-ups and symplectic resolutions.
\end{enumerate}
\end{examples}

\section{$K$-contact and Sasakian groups} \label{fund-group}

Following Boyer and Galicki \cite{BG} we will call a group $\Gamma$ to be
{\em $K$-contact} if it can be realized as the fundamental group of a 
compact $K$-contact manifold, and we will call $\Gamma$ to be a {\em 
Sasakian group} if there exists a compact Sasakian manifold $M$ with 
$\pi_1(M)\,\cong\,\Gamma$. In \cite{BG} the authors pose a problem of 
realizing finitely presented groups as fundamental groups of $K$-contact 
or Sasakian manifolds. Since Sasakian manifolds constitute an 
odd-dimensional counterpart of the class of K\"ahler manifolds, it is 
natural to expect that not all finitely presented groups are Sasakian.

On the other hand, we will show in this section that any finitely presented 
group is $K$-contact.

Using the analogy between K\"ahler and Sasakian 
manifolds, we ask the following question: {\it Can lattices in 
semisimple Lie groups be Sasakian?}

We show that the above question is related 
to the {\it orbifold fundamental groups} of K\"ahler orbifolds. Some 
restrictions of Sasakian groups were found in \cite{X}. In particular, it 
was proved that some lattices in semisimple Lie groups cannot be 
fundamental groups of regular Sasakian manifolds. Here we
strengthen this result showing that such groups cannot even be Sasakian.

\begin{proposition}\label{prop:fat} Let $X$ be any compact 
connected symplectic manifold. There exists a Boothby-Wang fibration
corresponding to some integral multiple
of the given symplectic form
$$
 S^1\,\lrightarrow\, P\,\lrightarrow\, X
 $$
such that the total space $M$ of the associated
fiber bundle
$$
 S^3\,\lrightarrow\, M\, :=\, P\times^{S^1}S^3\,\lrightarrow\, X
 $$
admits a $K$-contact structure, and also $\pi_1(M)\,\cong\,\pi_1(X)$.
\end{proposition}

\begin{proof} 
We know that Boothby-Wang fibrations are
fat with the moment map having nonzero values. The image of the
moment map consists of fat vectors (see Theorem \ref{thm:Boothby-Wang}).
Moreover, $S^3$ is $K$-contact, and the Hopf $S^1$-action preserves
the $K$-contact structure; this is straightforward, but one can also use a
general description of $K$-contact
manifolds given in \cite[Chapter 7]{BG} or in \cite{BMS}.
By Lerman's Theorem \ref{thm:lerman-contact}, any Boothby-Wang fibration yields
an associated bundle whose total space 
$M\,=\,P\times^{S^1}S^3$ admits a $K$-contact structure as well. Clearly,
we have $\pi_1(M)\,\cong\,\pi_1(X)$.
\end{proof}

\begin{theorem}\label{thm:main}
Any finitely presented group is $K$-contact.
\end{theorem}

\begin{proof}
A well-known result of Gompf \cite{Gompf} shows that any finitely presentable group $\Gamma$ can be 
realized as the fundamental group of some closed symplectic manifold of dimension 
$\geq 4$. Therefore, the theorem follows using Proposition \ref{prop:fat}.
\end{proof}

We now recall a theorem of Margulis from \cite{M}.

\begin{theorem}[\cite{M}]\label{thm:margulis}
Let $\mathcal G$ be a connected semisimple Lie group of rank at least
two and with no co-compact factors. Let $\Gamma\,\subset\,
\mathcal G$ be an irreducible arithmetic lattice in $\mathcal G$, and
let $\Sigma$ be a normal subgroup of $\Gamma$. Then either $\Sigma\,\subset\,
Z({\mathcal G})$ (the center of $\mathcal G$), or $\Gamma/\Sigma$ is finite.
\end{theorem}

\begin{proposition}\label{prop:lattices1} Let $\Gamma$ be an irreducible 
arithmetic lattice in a semisimple real Lie group $\mathcal G$ of rank at least two with
no co-compact factors and with trivial center. If $\Gamma$ is Sasakian, then
it must be isomorphic to the group $\pi_1^{orb}(M)$ of some K\"ahler
orbifold. Moreover, $\Gamma$ cannot be a co-compact arithmetic lattice
in $SO(1,n)$, $n>2$, or $F_{4(20)}$, or a simple real non-Hermitian Lie
group of noncompact type with real rank at least $20$.
\end{proposition}

\begin{proof}
Assume that $\Gamma$ is a Sasakian group. Let $\Gamma\,\cong\,\pi_1(X)$, 
where $X$ is a compact quasi-regular Sasakian manifold. We know that $X$ is
a rational circle
bundle over a compact K\"ahler orbifold $M$. Let 
$$S^1 \,\hookrightarrow \,X \,\stackrel{\pi}{\longrightarrow} \,M$$
be this orbi-bundle. By \cite[Theorem 4.3.18]{BG}, for any orbi-bundle
obtained by an action of a Lie group $G$
$$ 
G\,\lrightarrow\, X\,\lrightarrow\, M\, ,
$$
there is a long homotopy exact sequence
$$
 \cdots\, \lrightarrow\,\pi_i(G)\,\lrightarrow\, \pi_i^{orb}(X)\,\lrightarrow\,
\pi_i^{orb}(M)\,\lrightarrow \,\pi_{i-1}(G)\,\lrightarrow\, \cdots 
$$
Since in our case $G\,=\,S^1$, we have the exact sequence
 $$
 \cdots \,\lrightarrow\,\mathbb{Z}\,\lrightarrow\,\Gamma
\,\lrightarrow\,\Gamma'\,\lrightarrow\, \{1\}\, ,
 $$
where $\Gamma'\,=\,\pi_1^{orb}(M)$. Now by Theorem \ref{thm:margulis}, we have
$\Gamma\,\cong\,\pi_1^{orb}(M)$, because the image of $\mathbb{Z}$ must be in
the center of $\mathcal G$, which is trivial by the
assumption on $\mathcal G$. Since there is also a surjection
$\pi_1^{orb}(M)\,\lrightarrow\, \pi_1(M)$, we get a surjection 
 $$
 h\,:\,\Gamma\,\lrightarrow\,\pi_1(M)\, .
 $$
Consider the locally symmetric Riemannian space $B\,=\,
\Gamma\setminus {\mathcal G}/K$, where $K$ is a maximal compact subgroup of
$\mathcal G$. Then we have
$\Gamma\,\cong\,\pi_1(B)$. Since the sectional curvature of $B$ is non-positive, we
have the following:
\begin{enumerate}
 \item there is a harmonic map $f\,:\, X\,\lrightarrow\, B$ such that
$f_*\,:\,\pi_1(X)\,\lrightarrow\, \pi_1(B)$ 
 is an isomorphism (Eells-Sampson theorem \cite{ES}).
\item $f_*(\xi)\,=\,0$ for the Reeb vector field $\xi$ on $X$ (see \cite{P}).
\end{enumerate}
The above statement (2) shows that $f$ is constant on orbits of the Reeb vector field.
The fibers of $p$ are circles. Therefore, $f$ factors through $M$, meaning,
there exists $g\,:\, M\,\lrightarrow\, B$ such that $g\circ p\,=\,f$.
Therefore, on the level of fundamental groups one obtains $g_*\circ p_*\,=\,
f_*$, which shows that there is also a surjection 
 $$
 g_*\,:\,\pi_1(M)\,\lrightarrow\,\pi_1(B)\,=\,\Gamma\, .
 $$
Hence there is a sequence of two surjections
$$
\CD
\Gamma @>{h}>>\pi_1(M) @>{g_*}>>\Gamma
\endCD\, .
$$

Recall that $\Gamma$ is an arithmetic lattice in a semisimple Lie group
with trivial center. It means that 
$g_*\circ h$ cannot have non-trivial
kernel (again, by Theorem \ref{thm:margulis}). Hence, $h$ cannot have a
non-trivial kernel as well, and, therefore, it must be an isomorphism.
So $\pi_1(M)\,\cong\,\Gamma$.

It is 
known that the fundamental group of the topological space underlying a
K\"ahler orbifold is K\"ahler.
Indeed, a resolution of the singularity of the underlying space
of a K\"ahler orbifold is a compact
K\"ahler manifold. On the other hand, by \cite[p. 203, Theorem (7.5.2)]{Ko},
this resolution of singularity gives an isomorphism of fundamental groups
because the singularities are quotient of a smooth variety by the action
of a finite group. Therefore, the fundamental group of the
topological space underlying a K\"ahler orbifold is K\"ahler.

Thus, $\Gamma$ must be K\"ahler. But the latter is impossible, by a result
of Carlson and Hern\'andez \cite{CH}.
\end{proof}

\section*{Acknowledgements}

We are very grateful to the referee for pointing out an error in an earlier
proof of Proposition \ref{prop:Leschetz-higher-massey} (see Remark \ref{ref-rem}).
We also thank the referee for very useful suggestions.
We would like to thank John Oprea and Francisco Presas for useful comments.
The first author is supported by the J. C. Bose Fellowship. The second 
author was partially supported through Project MICINN (Spain) 
MTM2011-28326-C02-02, and Project of UPV/EHU ref.\ UFI11/52. The third 
author was partially supported by Project MICINN (Spain) MTM2010-17389. 
The fourth author was partly supported by the ESF Research Networking 
Programme CAST (Contact and Symplectic Topology).

\end{document}